\documentclass{amsart}

\usepackage{array}
\usepackage{color}

\begin{document}

\newcommand{\Pf}{{\em Proof}. }
\newcommand{\EPf}{\hfill$\square$}

\newcommand{\Lg}{\mbox{$\mathfrak g$}}
\newcommand{\Ll}{\mbox{$\mathfrak l$}}
\newcommand{\Lh}{\mbox{$\mathfrak h$}}
\newcommand{\Lk}{\mbox{$\mathfrak k$}}
\newcommand{\Lp}{\mbox{$\mathfrak p$}}
\newcommand{\La}{\mbox{$\mathfrak a$}}
\newcommand{\Lb}{\mbox{$\mathfrak b$}}
\newcommand{\Lm}{\mbox{$\mathfrak m$}}
\newcommand{\Ln}{\mbox{$\mathfrak n$}}
\newcommand{\Lc}{\mbox{$\mathfrak c$}}
\newcommand{\Lt}{\mbox{$\mathfrak t$}}
\newcommand{\Lq}{\mbox{$\mathfrak q$}}
\newcommand{\Ls}{\mbox{$\mathfrak s$}}
\newcommand{\Lz}{\mbox{$\mathfrak z$}}

\newcommand\liegr{\sf}

\newcommand{\SU}[1]{\mbox{${\liegr SU}(#1)$}}
\newcommand{\SUxU}[2]{\mbox{${\liegr S(U}(#1)\times{\liegr U}(#2))$}}
\newcommand{\U}[1]{\mbox{${\liegr U}(#1)$}}
\newcommand{\SP}[1]{\mbox{${\liegr Sp}(#1)$}}
\newcommand{\SO}[1]{\mbox{${\liegr SO}(#1)$}}
\newcommand{\SOxO}[2]{\mbox{${\liegr S(O}(#1)\times{\liegr O}(#2))$}}
\newcommand{\OG}[1]{\mbox{${\liegr O}(#1)$}}
\newcommand{\Spin}[1]{\mbox{${\liegr Spin}(#1)$}}
\newcommand{\G}{\mbox{${\liegr G}_2$}}
\newcommand{\F}{\mbox{${\liegr F}_4$}}

\newcommand\fieldsetc{\mathbb}

\newcommand{\Z}{\fieldsetc{Z}}
\newcommand{\R}{\fieldsetc{R}}
\newcommand{\C}{\fieldsetc{C}}
\newcommand{\Q}{\fieldsetc{H}}
\newcommand{\Ca}{\fieldsetc{O}}
\newcommand{\K}{\fieldsetc{K}}

\newtheorem*{thm}{Theorem}
\newtheorem{cor}{Corollary}[section]
\newtheorem{prop}{Proposition}[section]
\newtheorem{lem}{Lemma}[section]

\renewcommand{\thefootnote}{$\blacksquare$\arabic{footnote}$\blacksquare$}
\newcommand{\ft}{\footnote}

\pagenumbering{arabic}
\pagestyle{plain}

\title{The curvature of orbit spaces}

\author[C.~Gorodski]{Claudio Gorodski}
\author[A.~Lytchak]{Alexander Lytchak}

\address{Instituto de Matem\'atica e Estat\'\i stica, Universidade de
S\~ao Paulo, Rua do Mat\~ao, 1010, S\~ao Paulo, SP 05508-090, Brazil}

\email{gorodski@ime.usp.br}

\address{Mathematisches Institut, Universit\"at zu K\"oln,
Weyertal 86-90, 50931 K\"oln, Germany}

\email{alytchak@math.uni-koeln.de}

\thanks{The first author has been partially supported 
by the CNPq grant 303038/2013-6 and the FAPESP project 2011/21362-2.}

\date{\today}

\begin{abstract}
We investigate orbit spaces of isometric actions on unit spheres
and find a universal upper bound for the infimum of their  
curvatures.  
\end{abstract}

\maketitle

\section{Introduction}

Let $G$ be a compact Lie group acting by isometries on the unit sphere $S^n$.
The space of orbits  $X=S^n/G$  is an Alexandrov space of curvature at
least $1$ and diameter at most $\pi$ with respect to the natural quotient metric. The following question of K.~Grove has been investigated    in  \cite{McG,Gr, GS,Searlediameter} and remains widely open in general:
\begin{quote}
\it
How small can the diameter of the orbit space~$X$ be?
\end{quote}

The group $G$ can act transitively on $S^n$,
in which case $X$ is a point and has diameter zero. Also, the quotient
spaces~$S^1/G$ for  large cyclic groups $G$ have arbitrary small nonzero
diameters.  On the other hand, in has been shown in~\cite{Gr} that
for any fixed $n\geq 2$ there is a gap theorem:
for some positive $\epsilon(n)$, the diameter of any  $X=S^n/G$ is at
least $\epsilon(n)$ if~$X$ is not a point.
It seems plausible that a universal, dimension-independent bound
$\epsilon $ should exist, and the analysis of special classes of actions
suggests that such an $\epsilon$ might be rather large~\cite{McG,Gr}.
The case $X$ is $1$-dimensional is well 
understood~\cite[Theorem~B]{Searlediameter}.  In this paper, we consider
the case $\dim(X)\geq2$ and provide an answer to the following closely 
related question:
\begin{quote}
\it
How curved can the  orbit space~$X$ be?
\end{quote}

Denote by  $\kappa _X$ the largest number $\kappa$ such that the
orbit space $X$ is an Alexandrov space of curvature $\geq \kappa$.  
Note that  $\kappa _X$ equals the infimum of the Riemannian
sectional curvatures in the regular part of $X$, see 
Subsection~\ref{firstsubsec}.
Due to the theorem of Bonnet-Myers, $\kappa_X$ provides  an \textit{upper}
bound for the diameter, namely $\mathrm{diam}(X)\leq \pi/\sqrt{\kappa_X}$.
Therefore the existence of a uniform upper bound for $\kappa_X$  is 
 necessary   for the existence of a
uniform lower bound for the diameter. 
 The main result of this paper confirms this necessary condition and shows 
that~$\kappa_X\leq4$:

\begin{thm}\label{main}
Let $X=S^n/G$ be the orbit space of an isometric action of a 
compact Lie group $G$ on the unit sphere $S^n$ and assume 
that $\dim(X)\geq2$. 
If $X$ is an Alexandrov space of  curvature $\geq \kappa$, then
$\kappa\leq 4$.
\end{thm}

This result is sharp  as the Hopf action of the circle on $S^3$ shows.
Moreover,  there are families of actions for which $X$
an orbifold of constant curvature~$4$, see \cite{GL3} for their
classification.
 On the other hand, one can see that  for most
actions the infimum of the sectional curvatures  $\kappa _X$
of the orbit space $X$ is equal to  $1$. It seems to be possible,
but much more technical, to determine all actions for which
the corresponding number $\kappa_X$ is larger than $1$.
We hope to address this classification problem in a forthcoming work.
To put Theorem~\ref{main}  in context,
recall that the \emph{supremum} of sectional curvatures in the 
regular part of~$X$
can be infinite, and it is finite
precisely in case $X$ is a Riemannian orbifold~\cite{LT}; such actions
are classified in \cite{GL3}. 
We mention that the size and general structure of
orbit spaces of unit spheres  as in Theorem \ref{main}
may  have applications to isometric actions on general Riemannian
manifolds, cf.
\cite{GroveSearle,grovewilking} and the references therein.

The proof of Theorem~\ref{main}
uses few simple ideas. Strata in the orbit space are locally convex,
thus inheriting the lower curvature bound from their ambient space.
On the other hand, any stratum is contained in the principal stratum of the
orbit space of another  isometric action on a unit sphere,
as has been observed in~\cite[\S5.1]{GL3}.  This allows for an inductive approach
to the problem and reduces the question to the case where no singular strata
of dimension larger than  $2$ are present. But the absence of such strata
implies that the rank of the original group is at most $3$.
The remaining cases are excluded by an index comparison argument and, in final
instance, by the classification of irreducible representations of compact
simple Lie groups.

It is an interesting question if our theorem has a geometric explanation
not relying on the classification of representations, and if it can be
extended to the case of singular Riemannian foliations on the unit sphere.
Recently, a large family of singular Riemannian foliations
has been constructed in \cite{Radeschi2},
most of which are inhomogeneous (see also~\cite{GorRad}).
All quotient spaces $X$  arising from these foliations are
Alexandrov spaces with curvature $\geq 4$,
but always have some tangent planes with  curvature equal to $4$.

The authors would like to thank Marco Radeschi for interesting discussions
and useful comments. 

\section{Preliminaries}
\subsection{Orbit spaces and strata} \label{firstsubsec}

We recall some basic results about orbit spaces and refer, for instance,
to~\cite[\S2]{GL3} for more details. An isometric action of a
compact Lie group $G$ on a unit sphere is the restriction
of an orthogonal representation $\rho$ of $G$ on an
Euclidean space $V$.
If $S(V)$ denotes the unit sphere of $V$ and $X=S(V)/G$, then
the cohomogeneity $c(\rho)$ of $\rho$ satisfies $c(\rho)=\dim (X)+1$.
Let $\kappa_X =\kappa _{\rho}$ be defined as in the introduction.

The orbit space~$X$ is an Alexandrov space stratified by smooth
Riemannian manifolds, namely, the sets of orbits with the same isotropy
groups up to conjugation. Any such  stratum $Y$ is a locally convex subset
of $X$, so that the infimum of the sectional curvatures along $Y$ is also
bounded below by $\kappa_X$.  There is exactly one maximal stratum,
the set of principal orbits  $X_{\mathrm{reg}}$ of $X$,
corresponding to the unique minimal conjugacy class of isotropy groups.
The corresponding isotropy groups are called the principal isotropy groups.
The set of regular points $X_{\text{reg}}$  is open, dense and convex in $X$.
The restriction of the projection to
the regular set $S^n_{\text{reg}}\to X_{\text{reg}}$ is 
a Riemannian submersion.
Moreover, the orbit  space $X$ is the completion of the convex
open submanifold $X_{\text{reg}}$. Hence Toponogov's globalization theorem,
for instance the version in~\cite{Petglob}, shows that $\kappa _X$ is equal to
the infimum of sectional curvatures of $X_{\text{reg}}$. 

\subsection{Strata and rank}
Denote the rank of the group $G$ by~$k$.
Then there exists a point $p\in S(V)$ such that the isotropy group $G_p$
has rank at least~$k-1$ cf.~\cite[Lemma~6.1]{wilking-acta}.
We infer:

\begin{lem} \label{minorbit}
Let a group $G$ act by isometries on $S(V)$.
If $G$ has rank $k$, then the minimal dimension $\ell$
of a $G$-orbit is at most $\dim (G)-k+1$.
\end{lem}

We will need another simple observation:

\begin{lem} \label{strexist}
Let a group $G$ of rank $k$ act by isometries on $S(V)$.
If the $G$-action has trivial principal isotropy groups, then
$X=S(V)/G$ contains a non-maximal stratum of dimension at least $k-2$.
\end{lem}

\begin{proof}
We proceed by induction on $k$.
In the initial case $k=2$, there exists a point with non-trivial
isotropy group, so the quotient $X$ contains non-regular points, and
therefore at least one non-maximal stratum (of dimension at least $0$).
In general, we first find a point $p\in S(V)$ such that the
isotropy group $G_p$ has rank at least $k-1$.
The slice representation of $G_p$ on the normal space $\mathcal H_p$
has again trivial principal isotropy groups, so the inductive
assumption yields that the quotient of the unit sphere $S(\mathcal H_p)$
by~$G_p$ has a non-regular stratum of dimension at least  $k-3$.
The corresponding stratum in $\mathcal H_p/G_p$, and thus also
in a neighborhood of $x=G\cdot p$ in $X$, has dimension at least $k-3+1=k-2$,
which completes the induction step.
\end{proof}

\subsection{Enlarging group actions and polar representations}
Let $\tau:H\to\OG V$ be a representation of a compact Lie group. 
Consider a closed subgroup~$G$  of~$H$ and the 
representation
$\rho:G\to \OG V $ obtained by restriction.
The canonical projection $S(V)/G \to S(V)/H$ restricts to 
a Riemannian submersion on an open and dense set, 
hence this map does not decrease the sectional curvatures
by the formula of O'Neill~\cite[Corollary~1.5.1]{GW}. This shows:

\begin{prop}\label{enlargement}
Suppose an orthogonal representation $\rho:G\to\OG{V}$
is the restriction of another representation $\tau:H\to\OG{V}$,
where $G$ is a closed subgroup of $H$.
If $c(\tau )\geq 3$  then $\kappa_\rho \leq \kappa_\tau$.
\end{prop}

Recall that polar representations $\rho:G\to \OG V$ are
exactly those whose induced action on $S(V)$ has orbit space
of constant curvature $1$~\cite[Introd.]{GL2}. In particular,
all polar representations $\tau$ with $c(\tau ) \geq 3$ satisfy $\kappa _{\tau} =1$.  We now have:

\begin{cor} \label{tensor}
Let $\rho_i:G_i\to\OG{V_i}$ for $i=1$, $2$
be $\mathbb F$-linear representations, where $\mathbb F=\R$, $\C$, $\Q$.
If $\dim_{\mathbb F}(V_i)\geq3$ for $i=1$, $2$,
then the  tensor product representation  $\rho$ of $G=G_1\times G_2 $ on
$V=V_1\otimes_{\mathbb F}V_2$ satisfies $\kappa _{\rho} =1$.
\end{cor}

\begin{proof}
It follows directly from Proposition~\ref{enlargement} by
 taking $H=\OG{V_1}\otimes\OG{V_2}$,
$H=\U{V_1}\otimes\U{V_2}$, $H=\SP{V_1}\otimes\SP{V_2}$ if
$\mathbb F=\R$, $\C$, $\Q$, respectively.   Indeed, the representation
$\tau$ of $H$ on $V$ is polar and enlarges $\rho$.
Moreover,  the cohomogeneity $c(\tau)$ is the minimum of the
two numbers $\dim_{\mathbb F}(V_i) \geq 3$.
\end{proof}

\subsection{Reductions and strata as subquotients}
Let a compact Lie group $G$ act by isometries on
an unit sphere $S(V)$. Fix an arbitrary point  $p\in S(V)$, denote
its isotropy group by $G_p$, consider its orbit $x=G\cdot p\in X=S(V)/G$,
and denote by $Y$ the stratum of $X$ which contains $x$.
The set of all points in $V$ fixed by $G_p$ is a subspace $W$
on which the normalizer $N$ of $G_p$ in $G$ acts isometrically.
There is a canonical map $S(W)/N\to X=S(V)/G$,
which is  length-preserving and  a local isometry on an open dense subset
of $S(W)/N$~\cite[Lemma~5.1]{GL3}. Moreover, $Y$ is dense in the image of
this map. We deduce that the infimum
of the curvatures on $S(W)/N$ is also bounded from below by $\kappa_X$.
Finally, note that $G_p$ acts trivially on $W$ and that the induced action of
$H=N/G_p$ on $S(W)$ has trivial principal isotropy groups. This shows:

\begin{prop}\label{substrata}
Let $\rho:G\to \OG V$ be a representation with $c(\rho )\geq 3$.
Let $X=S(V)/G$ and assume that $X$ has a singular stratum $Y$ of
dimension $d\geq 2$. Then there exists a representation $\tau:H\to\OG W$
with trivial principal isotropy groups such that $S(W)/H$ has dimension $d$
and $\kappa _{\rho} \leq \kappa _{\tau}$.
\end{prop}


\subsection{An index  estimate}
Let $G$ act on $S(V)$ as above and
consider the restriction $\pi$ of the projection $S(V)\to X$
to the regular part. Applying O'Neill's formula
for the curvature of the
Riemannian submersion $\pi$, we see that planes with curvature~$1$ exist in
$X=S(V)/G$ if and only if O'Neill's tensor of $\pi$ vanishes at some pair of
linearly independent, horizontal vectors. Therefore,
$\kappa _{\rho} >1$ directly implies that the dimension of a regular orbit
$G\cdot p$ is at least $\dim (X)-1$. On the other hand,
using index estimates one can say slightly more.

Indeed let $F=G\cdot p$ be a principal orbit and let
$\gamma  :[0,\pi] \to S^n$ be a unit speed $F$-geodesic, namely,
a geodesic starting perpendicularly to $F$. The
index of $\gamma$ as $F$-geodesic
(which equals the sum of focal multiplicities of $F$ along $\gamma$)
is equal to~$\dim (F)$.
On the other hand,  the index of $\gamma$ is also obtained as the sum of
the vertical index $\mathrm{ind}_{\mathrm{vert}} (\gamma)$ and the
horizontal index $\mathrm{ind}_{\mathrm{hor}} (\gamma)$,
see~\cite[Lemma 5.1]{LT} and \cite[\S3]{LytJacobi}.
The vertical index  counts the intersections of $\gamma$ with singular orbits:  $$\mathrm{ind}_{\mathrm{vert}} (\gamma) = \sum_{t\in (0,\pi)} (\dim (F) - \dim (G\cdot \gamma (t) )).$$
Moreover, the horizontal index is the index of the transversal Jacobi equation defined by Wilking in~\cite{Wilk}. This Jacobi equation has the
form $J''+R_t(J) =0$ for a time-dependent symmetric operator
$R_t :U\to U$ on a Euclidean vector space $U$ of dimension
$\dim (X)-1$. Around a regular point $\gamma(t)$,
the  Jacobi equation  $J''+R_t(J) =0$ is just the Jacobi equation in the
Riemannian manifold $X_{\mathrm{reg}}$ along the projected geodesic
$\pi\circ\gamma$. Therefore, if all sectional curvatures at regular points of
$X$ along $\gamma$ are at least $\kappa$, then $R_t \geq \kappa \cdot Id$.
Thus the standard index comparison~\cite[Lemma~2.6.1]{Kl} implies that
 $\mathrm{ind}_{\mathrm{hor}} (\gamma)\geq \dim (X)- 1$ in
case $\kappa _{\rho} >1$, and
$\mathrm{ind}_{\mathrm{hor}} (\gamma )\geq 2(\dim (X)- 1)$
in case $\kappa _{\rho} >4$.

Now we can easily deduce:

\begin{prop} \label{mainlem}
Let $\rho:G\to \OG V$ be as above.
Denote by~$\ell$ the smallest dimension of a $G$-orbit in $S(V)$, and
by~$m\geq 2$ be the dimension of the orbit space $X=S(V)/G$. Then:
\begin{enumerate}
\item if $\kappa >1$,  then $\ell\geq m-1$;
\item if $\kappa >4$, then $\ell\geq 2(m-1)$.
\end{enumerate}
%
%
\end{prop}

\Pf
Let $L$ be an orbit of smallest dimension $\ell$.  Take a regular orbit $F$ and
a horizontal unit speed
geodesic $\gamma:[0,\pi] \to S(V)$ starting in $F$ and going
through~$L$. The index of $\gamma$ is $\dim (F)$. On the other hand,
the vertical index is at least $\dim (F) -\ell$.  Since the index of $\gamma$
is the sum of the vertical and the horizontal indices,
we see that $\ell$ cannot be smaller than the horizontal index of $\gamma$.
Thus the result follows from the index estimates above.
 \EPf

\section{Main result}
\subsection{Formulation}
In this section, we prove Theorem~\ref{main}. Suppose to the contrary that
there exists a representation $\rho:G\to\OG V$  a compact Lie group $G$
such that $X=S(V)/G$ has dimension~$m\geq 2$
and satisfies $\kappa _\rho >4$.  We may assume that $m$ is minimal among all
such examples. Namely, for any representation $\tau :H\to\OG W$ of a
compact Lie group $H$ such that the dimension~$m'$ of the
orbit space $Y=S(W)/H$ satisfies
$m>m'\geq 2$, we have $\kappa _{\tau} \leq 4$.
We assume further that, for this~$m$,
$g:=\dim (G)$ is minimal among all such examples, and
we fix the representation $\rho$ throughout the proof.

\subsection{Principal isotropy and identity component}
By the assumption on the minimality of $g$, the representation of
$\rho$ is \emph{reduced} in the sense of \cite[\S1.2]{GL}:  for any other
representation $\tau:H\to\OG W$ such that $S(W)/H$ is isometric to $X$,
we have $\dim (H) \geq g$. In particular, this implies that the action of
$G$ on $S(V)$ has trivial principal isotropy groups.

Let $\rho _0$ be the restriction of $\rho$ to the identity
component $G^0$ of $G$.
Then the projection $X_0=S(V) /G^0 \to X=S(V)/G$ is a Riemannian covering
over the set of regular points of $X$. We deduce
$\kappa _{\rho} =\kappa _{\rho_0}$.  Hence, we may replace $G$ by $G^0$
and assume from now on that $G$ is connected.

\subsection{Strata and rank}
The minimality assumption on $m$ together with Proposition~\ref{substrata} show
that $X$ does not contain singular strata of dimensions larger than $1$.
From   Lemma~\ref{strexist}
we deduce that the rank~$k$ of~$G$ is at most~$3$.

\subsection{Irreducibility}
Our assumption yields $\mathrm{diam}(X)\leq\pi/\sqrt{\kappa_X}<\pi/2$.
This implies that the representation $\rho:G\to\OG V$ is 
irreducible~\cite[\S5]{GL}.


\subsection{Basic identity and inequality}
Now~$\rho$ is  an irreducible representation of a connected group $G$ of
dimension~$g$ and rank~$1\leq k\leq 3$. The representation is also
reduced, so the principal isotropy groups are
trivial and the dimension $n$ of $S(V)$ satisfies
$$g+m=n.$$
From Lemma~\ref{minorbit} and Proposition~\ref{mainlem} we deduce
$$g-k\geq 2m-3.$$
In particular, we have $n\leq\frac 3 2 g+ 1$.

\subsection{The case $m=2$}
Due to \cite{str}, in this case the only reduced
representations $\rho:G\to \OG V$ of a connected non-trivial group $G$
are representations of the circle group $\U 1$ of $\R^4$.
But such representations are reducible.

\subsection{The case $m=3$}
Due to the classification of irreducible representations of
cohomogeneity $4$ of connected compact Lie groups~\cite[Theorem~1.8]{GL},
in this case the only reduced representations
are given by the actions of $\SO 3$ on $\R ^7$ and by $\U 2$ on $\R ^8$.
Both cases contradict the inequality $g-k \geq 2m-3=3$.

\subsection{The case $m=4$}
Due to the classification of irreducible representations of
cohomogeneity $5$ of connected compact Lie groups~\cite[Theorem~1.8]{GL},
in this case the only reduced representations are given by the
action of $\SU 2$ on $\R ^8$ and by the action of
$\SO3\times \U 2$ on $\R ^{12} =\R^3\otimes _{\R} \R^4$.
Both cases contradict the inequality
$g-k \geq 2m-3=5$.

\subsection{Consequences}
Henceforth we may  assume that $m\geq 5$.
Therefore $g\geq 2m-3+k\geq7+k$. This immediately excludes
 the possibility that $G$ have rank $1$ or that $G$ be covered by a product
of (two or three) groups of rank $1$.
Moreover, $G$ cannot be covered by $\SU 3$ or $\U 1 \times \SU 3$.

\subsection{Type of representation}
We claim that the  centralizer $N$ of $\rho(G)$ in $\OG V$ has $\rho(G)$ as its
identity component. Otherwise, we find a subgroup $H$ of $\OG V$
containing $\rho(G)$ with one dimension more. The inclusion $\tau:H\to \OG V$
is an enlargement of $\rho$. The quotient space $Y=S(V) /H$ has dimension at
least $m-1\geq4$, and $\kappa _{\rho} \leq \kappa _{\tau}$
owing to Proposition~\ref{enlargement}.
By our minimality assumption on $m$, we must have $\dim(Y)=\dim(X)$.
But then $\tau $ and $\rho$ have the same orbits. This implies
that $\tau$ has non-trivial principal isotropy groups, so it cannot be
reduced.
It follows that neither $\rho$ is reduced, in contradiction to our assumption.

We deduce that $\rho$ cannot be of quaternionic type,
and if $\rho$ is of complex type then $G$ is covered by
$\U 1\times G'$ for some connected compact Lie group $G'$.

\subsection{Representations of complex type}
Assume that $\rho$ is of complex type, namely,
$V$ admits an invariant complex structure;
in particular, $\dim(V)$ is even.
It follows from the preceding two subsections that
$G$ must be covered by  $\U 1\times G'$, where $G'$ a connected compact
simple Lie group of rank $2$ and $G'\neq \SU 3$. Hence $G'=\SP2 $ or $\G$.

If $G'=\SP2$, we have $g=11$ and $11\geq2m$, which gives~$m=5$.
Then the dimension $n+1$ of $V$ is $17$, which is odd and thus impossible.

If $G'=\G$, we obtain $g=15$, hence $15\geq2m$.
Since the dimension of $V$ is even, we deduce that $m$ must be even as well,
hence $m=6$. Now $V$ has dimension
$g+m+1=22$.  But there are no irreducible representations of $\G$
on $\C^{11}$ (cf.\ appendix in~\cite{K}).

There remains only the case $\rho$ is of real type.

\subsection{Non-simple groups}
If  $G$ is not simple then, up to a finite covering, it must have the
form  $G=G_1\times G_2$ where $\mathrm{rk}(G_1)=2$ and $\mathrm{rk}(G_2)=1$
and where $G_1$ is simple. Now $G_1$ is one of~$\SU 3$, $\SP2 $ or $\G$.
Since $\rho$ is irreducible and there is no $G$-invariant complex structure on $V$, $G_2$ cannot be $\U 1$, so it is covered by $\SU 2$.
The representation $\rho$ is a tensor product $\rho =\rho_1 \otimes _{\mathbb F} \rho _2$ where the $\rho_i :G_i\to\OG{V_i}$ are irreducible
$\mathbb F$-linear representations and $\mathbb F =\R$, $\C$ or $\Q$.
Since the representation $\rho$ is of real type, we cannot
have~$\mathbb F= \C$. From Corollary~\ref{tensor},
we see that the $\mathbb F$-dimensions of $V_i$ cannot be both larger
than~$3$. We deduce that $\mathbb F$ cannot be $\R$ either,
thus $\mathbb F=\Q$.

Note that $\SU 3 $ and $\G$  do not have  irreducible
representations of quaternionic type. We are left with the case $G_1=\SP 2$,
$G_2=\SU 2$.  Then $g=13$ and $13\geq2m$.  Since the  dimension  $g+m+1$ of $V$ must be divisible by $4$, we get $m=6$ and $\dim(V)=20$.
Then $20=4rs$, where $r$ and $s$ are the dimensions over $\mathbb H$ of $V_1$ and $V_2$, respectively. We deduce $r=5$ and $s=1$.
However, there does not exist irreducible  quaternionic representation of
$\SP 2$ on $\mathbb H$ or $\mathbb H^5$ (cf.~again the appendix in~\cite{K}).

\subsection{Kollross' table and the case of simple groups}
We are left with the case $G$ is a simple group of rank $2\leq k\leq 3$.
The dimension $n+1$ of $V$, thus the degree of the corresponding
complexified representation satisfies
$n+1\leq \frac 3 2 g +2  < 2g+2$. Thus we may apply
Lemma~2.6 from~\cite{K} to deduce that $\rho$ is one of the representations
listed in the tableau therein. Only four
of those representations are representations of groups of rank
$2\leq k\leq 3$,
namely two representations of $\G$ and two representations of $\Spin 7$.
Two of  these representations satisfy $g\geq \dim (V)$ (they 
also have cohomogeneity one),
which is impossible under our assumptions.
For the two remaining representations
the condition $n\leq \frac 3 2 g +1$ is violated.
This finishes the proof of Theorem~\ref{main}.

\providecommand{\bysame}{\leavevmode\hbox to3em{\hrulefill}\thinspace}
\providecommand{\MR}{\relax\ifhmode\unskip\space\fi MR }
\providecommand{\MRhref}[2]{%
  \href{http://www.ams.org/mathscinet-getitem?mr=#1}{#2}
}
\providecommand{\href}[2]{#2}

\end{document}